\documentclass{amsproc}
\usepackage{hyperref}
\usepackage{amssymb}
\usepackage{amsrefs} 
\usepackage{mathrsfs} 

\newtheorem{theorem}{Theorem}[section]
\newtheorem{proposition}[theorem]{Proposition}
\newtheorem{lemma}[theorem]{Lemma}
\theoremstyle{remark}
\newtheorem{remark}[theorem]{Remark}

\numberwithin{equation}{section}

\DeclareMathOperator{\Aut}{Aut}
\DeclareMathOperator{\divisor}{div}
\DeclareMathOperator{\End}{End}
\DeclareMathOperator{\Jac}{Jac}
\DeclareMathOperator{\Norm}{Norm}
\DeclareMathOperator{\Tr}{Tr}

\newcommand{\CC}{{\mathbf{C}}}
\newcommand{\FF}{{\mathbf{F}}}
\newcommand{\PP}{{\mathbf{P}}}
\newcommand{\QQ}{{\mathbf{Q}}}
\newcommand{\ZZ}{{\mathbf{Z}}}

\newcommand{\scrC}{{\mathscr C}}
\newcommand{\scrL}{{\mathscr L}}
\newcommand{\scrO}{{\mathscr O}}

\newcommand{\kbar}{\overline{k}}
\newcommand{\qbar}{\overline{q}}
\newcommand{\rbar}{\overline{r}}
\newcommand{\xbar}{\overline{x}}
\newcommand{\zbar}{\overline{z}}
\newcommand{\betabar}{\overline{\beta}}
\newcommand{\pibar}{\overline{\pi}}

\newcommand{\Kp}{K^+}
\newcommand{\col}{\,{:}\,}

\newcommand{\Drinfeld}{Drinfel{\cprime}d}
\newcommand{\Vladut}{Vl\u adu\c t}

\newcommand\lowtilde{\lower0.7ex\hbox{\textasciitilde}}

\newcommand\us{\textunderscore}

\begin{document}

\title[Bounds for genus-$4$ curves]
{New bounds on the maximum number of points on genus-$4$
 curves over small finite fields}

\author{Everett W. Howe}
\address{Center for Communications Research,
         4320 Westerra Court,
         San Diego, CA 92121-1967, USA.}
\email{however@alumni.caltech.edu}
\urladdr{\href{http://www.alumni.caltech.edu/~however/}
          {http://www.alumni.caltech.edu/\lowtilde{}however/}}

\date{8 March 2012}
\dedicatory{Dedicated to the memory of my mother, Norma Howe}



\copyrightinfo{2012}{Institute for Defense Analyses}

\begin{abstract}
For prime powers $q<100$, we compute new upper and lower bounds 
on~$N_q(4)$, the maximal number of points on a genus-$4$ curve over a
finite field with $q$ elements.  We determine the exact value of 
$N_q(4)$ for $17$ prime powers $q$ for which the value was previously
unknown.
\end{abstract}

\maketitle

\section{Introduction}
\label{S:intro}

For every prime power $q$ and integer $g\ge 0$, let $N_q(g)$ denote
the maximal number of points on a curve of genus~$g$ over~$\FF_q$. In
this paper we compute new upper and lower bounds on $N_q(4)$ for prime
powers $q$ less than $100$, and we find the exact value of $N_q(4)$ for
$17$ prime powers $q$ for which the value was previously unknown.

Weil's famous `Riemann Hypothesis' for curves over finite 
fields~\cites{Weil1940,Weil1941,Weil1945,Weil1946} gives the
fundamental result that 
\[  N_q(g) \le q + 1 + 2g\sqrt{q},  \]
which generalizes Hasse's theorem for the number of points on an 
elliptic curve over a finite field~\cite{Hasse1936}.  Weil's bound was
improved significantly in the case where $g$ is large with respect to 
$q$ by Manin~\cite{Manin1981}, Ihara~\cite{Ihara1981}, and 
\Drinfeld\ and \Vladut~\cite{DrinfeldVladut1983}, and for fixed $q$ we
have the \Drinfeld-\Vladut\ bound
\[
   N_q(g)\leq (\sqrt{q}-1+o(1)) g \qquad\text{as $g \rightarrow \infty$.}
\]
When $q$ is a square this bound is asymptotically optimal, in the sense
that
\[ \limsup N_q(g)/g = \sqrt{q} - 1.\]
The value of this lim sup is not known for any nonsquare~$q$.

For any particular choice of $q$ and $g$, it can be a difficult
computational problem to determine the actual value of $N_q(g)$ and to
find a curve that attains this number of points. When $g$ is less than 
$(q - \sqrt{q})/2$, the best upper bound known is often Serre's 
improvement to the Weil bound~\cite{Serre1983a}:
\[ N_q(g) \le q + 1 + g \lfloor 2\sqrt{q}\rfloor. \]
But for every fixed genus $g>2$, nothing is known about the proportion
of prime powers for which this bound is attained.

For $g=1$ and $g=2$, the exact value of $N_q(g)$ can be calculated.
For $g=1$ this is due essentially to Deuring~\cite{Deuring1941}
(see~\cite{Waterhouse1969}*{Thm.~4.1, p.~536}), and for $g=2$ to
Serre~\cites{Serre1983a,Serre1983b,Serre1984} (see 
also~\cite{HoweNartRitzenthaler2009}).  But even for $g=3$ difficulties
arise, as is explained in~\cite{LRZ}.

The online tables available at \url{http://manypoints.org} collect the
best known upper and lower bounds on $N_q(g)$ for $g\le 50$ and for 
various values of $q$: the primes less than $100$, the prime powers 
$p^i$ for $p<20$ and $i\le 5$, and the powers of $2$ up to~$2^7$.  
Despite the lack of an explicit formula for $N_q(3)$, for all of the 
prime powers $q$ listed in these tables the value of $N_q(3)$ has been
calculated.  This leaves $N_q(4)$ as the next challenge. In this paper,
we give new upper and lower bounds for $N_q(4)$ for $22$ of the $23$ 
prime powers $q<100$ for which the exact value had not been previously
computed; we find the exact value of $N_q(4)$ for $17$ of these prime
powers.

Our results are given in Table~\ref{T:results}.  The entries in the
`old' column show the results that were listed on 
\href{http://manypoints.org}{\texttt{manypoints.org}} on 
1 August~2011.  Note that the entries in the `new' column show that
there are now only six values of $q$ less than $100$ for which the
exact value of $N_q(4)$ is not known.

\begin{table}
\renewcommand{\arraystretch}{1.25}
\begin{center}
\begin{tabular}{|r|r|r|r|r|r|r|r|r|r|r|}
\cline{1-3}\cline{5-7}\cline{9-11}
 $q$ & old & new &\qquad & $q$ & old & new & \qquad & $q$ & old & new \\
\cline{1-3}\cline{5-7}\cline{9-11}
 $2$ &            $8$ &            $8$ && $23$ &       --\,$58$  &            $57$  && $59$ &        --\,$118$ &            $116$ \\ 
 $3$ &           $12$ &           $12$ && $25$ &           $66$  &            $66$  && $61$ &        --\,$122$ & $118$\,--\,$119$ \\ 
 $4$ &           $15$ &           $15$ && $27$ &           $64$  &            $64$  && $64$ &            $129$ &            $129$ \\ 
 $5$ &           $18$ &           $18$ && $29$ &       --\,$70$  &  $67$\,--\,$68$  && $67$ &        --\,$132$ &            $129$ \\ 
 $7$ &           $24$ &           $24$ && $31$ &       --\,$73$  &            $72$  && $71$ & $132$\,--\,$136$ &            $134$ \\ 
 $8$ &           $25$ &           $25$ && $32$ & $71$\,--\,$72$  &  $71$\,--\,$72$  && $73$ &        --\,$139$ &            $138$ \\ 
 $9$ &           $30$ &           $30$ && $37$ &       --\,$84$  &            $82$  && $79$ &        --\,$148$ &            $148$ \\ 
$11$ & $33$\,--\,$34$ &           $33$ && $41$ &       --\,$90$  &            $88$  && $81$ &            $154$ &            $154$ \\ 
$13$ &       --\,$39$ &           $38$ && $43$ &       --\,$93$  &            $92$  && $83$ &        --\,$154$ &            $152$ \\ 
$16$ &           $45$ &           $45$ && $47$ &       --\,$100$ &            $98$  && $89$ &        --\,$162$ & $160$\,--\,$162$ \\ 
$17$ &       --\,$48$ &           $46$ && $49$ &       --\,$106$ & $102$\,--\,$106$ && $97$ &        --\,$174$ &            $174$ \\ 
\cline{9-11}
$19$ &       --\,$52$ & $48$\,--\,$50$ && $53$ &       --\,$110$ &            $108$ &  \multicolumn{4}{}{}                                        \\ 
\cline{1-3}\cline{5-7}
\end{tabular}
\end{center}
\vskip0.5em
\caption{Old and new ranges for $N_q(4)$, for $q<100$.}
\label{T:results}
\end{table}

We obtain our new upper bounds on $N_q(4)$ by using the results
of~\cite{HoweLauter2012}.  For each~$q$, we use the computer programs
associated with that paper to obtain restrictions on genus-$4$ curves 
over $\FF_q$ with many points.  Sometimes we learn that a curve with a
certain number of points must be a double cover of one of several
elliptic curves; in Section~\ref{S:covers} we show how to enumerate
such double covers.  Sometimes we learn that a curve with a certain 
number of points must correspond to a Hermitian lattice over a
quadratic order; we discuss these cases in Sections~\ref{S:maximal}
and~\ref{S:nonmaximal}. In two cases we find that the Jacobian of a 
curve with a certain number of points must have complex multiplication
by $\ZZ[\zeta_5]$; our analysis of Hermitian forms over this ring in 
Section~\ref{S:zeta} shows that such curves do not exist.  Our new
lower bounds come from explicit examples of curves, which we present
in Section~\ref{S:lower}.

We have implemented all of our calculations in the computer algebra
package Magma~\cite{magma}.  As we mentioned above, we discover
properties of genus-$4$ curves over $\FF_q$ with a given number of
points by using the programs associated with the
paper~\cite{HoweLauter2012}.  These programs are found in the package
\texttt{IsogenyClasses.magma}, which is available on the author's
website: Go to the bibliography page

\centerline{\href{http://alumni.caltech.edu/~however/biblio.html}
                {\texttt{http://alumni.caltech.edu/{\lowtilde}however/biblio.html}}
}

\noindent
and follow the link associated with the paper~\cite{HoweLauter2012}. 
The programs we use to enumerate double covers are also available
online, by starting at the URL given above and following the link 
associated with the present paper.

\section{Restrictions on genus-$4$ curves with many points}
\label{S:results}

In this section we present the results we obtained from running the
programs in \texttt{IsogenyClasses.magma}.

Table~\ref{T:upper} lists the $(q,N)$ pairs that we will have to
eliminate in order to prove our new upper bounds; for each $q$ and~$N$,
the table shows what \texttt{IsogenyClasses.magma} tells us about 
genus-$4$ curves over $\FF_q$ with $N$ points. An entry of 
``None exist'' means that \texttt{IsogenyClasses.magma} shows that no
genus-$4$ curve over $\FF_q$ has exactly $N$ points. Entries of the 
form ``Double cover of elliptic curve with trace $t$'' mean that any 
genus-$4$ curve over $\FF_q$ with $N$ points must be a double cover of
an elliptic curve over $\FF_q$ with trace $t$. An entry of
``Hermitian module over $R$'', where $R$ is an order in an imaginary 
quadratic field, means that any genus-$4$ curve over $\FF_q$ with $N$
points must have a Jacobian that is isogenous to the fourth power of
an ordinary elliptic curve over $\FF_q$ whose Frobenius endomorphism
generates the order $R$. An entry of 
``Hermitian module over $\ZZ[\zeta_5]$'' means that any genus-$4$ curve
over $\FF_q$ with $N$ points must have a Jacobian that is isogenous to
the square of an ordinary abelian surface  with complex multiplication
by the ring of integers of the $5$th cyclotomic field.

\begin{table}
\renewcommand{\arraystretch}{1.25}
\begin{center}
\begin{tabular}{|r|r|l|}
\hline
$q$  &   $N$ & Properties of a genus-$4$ curve over $\FF_q$ with $N$ points\\
\hline
$11$ &  $34$ & Hermitian module over $\ZZ[\zeta_5]$                        \\
$13$ &  $39$ & Double cover of elliptic curve with trace $-7$              \\
$17$ &  $48$ & Double cover of elliptic curve with trace $-8$              \\
$17$ &  $47$ & None exist                                                  \\
$19$ &  $52$ & Hermitian module over $\ZZ[\sqrt{-3}]$                      \\            
$19$ &  $51$ & None exist                                                  \\
$23$ &  $58$ & Double cover of elliptic curve with trace $-9$              \\
$29$ &  $70$ & Hermitian module over $\ZZ[2i]$                             \\          
$29$ &  $69$ & None exist                                                  \\
$31$ &  $73$ & Double cover of elliptic curve with trace $-11$             \\
$37$ &  $84$ & Double cover of elliptic curve with trace $-12$             \\
$37$ &  $83$ & None exist                                                  \\
$41$ &  $90$ & Hermitian module over $\ZZ[\sqrt{-5}]$                      \\           
$41$ &  $89$ & None exist                                                  \\
$43$ &  $93$ & Double cover of elliptic curve with trace $-13$             \\
$47$ & $100$ & Hermitian module over $\ZZ[\alpha]$                         \\
$47$ &  $99$ & None exist                                                  \\
$53$ & $110$ & Hermitian module over $\ZZ[2i]$                             \\          
$53$ & $109$ & None exist                                                  \\
$59$ & $118$ & Double cover of elliptic curve with trace $-15$             \\
$59$ & $117$ & Double cover of elliptic curve with trace $-15$             \\
$61$ & $122$ & Hermitian module over $\ZZ[\alpha]$                         \\
$61$ & $121$ & None exist                                                  \\
$61$ & $120$ & Double cover of elliptic curve with trace $-13$, or         \\
     &       & Hermitian module over $\ZZ[\zeta_5]$                        \\
$67$ & $132$ & Hermitian module over $\ZZ[\sqrt{-3}]$                      \\
$67$ & $131$ & None exist                                                  \\
$67$ & $130$ & Double cover of elliptic curve with trace $-14$             \\
$71$ & $136$ & Hermitian module over $\ZZ[\sqrt{-7}]$                      \\          
$71$ & $135$ & None exist                                                  \\
$73$ & $139$ & Double cover of elliptic curve with trace $-17$             \\
$83$ & $154$ & Double cover of elliptic curve with trace $-18$             \\
$83$ & $153$ & Double cover of elliptic curve with trace $-18$             \\
\hline
\end{tabular}
\end{center}
\vskip0.5em
\caption{What \texttt{IsogenyClasses.magma} tells us about genus-$4$
         curves over $\FF_q$ with $N$ points. Here $\zeta_5$ is a 
         primitive $5$th root of unity, $i = \sqrt{-1}$, and $\alpha$
         satisfies $\alpha^2 + \alpha + 5 = 0$.}
\label{T:upper}
\end{table}

\section{Double covers of elliptic curves}
\label{S:covers}

Table~\ref{T:upper} shows that there are a number of pairs $(q,N)$ such
that a genus-$4$ curve over $\FF_q$ with $N$ points must be a double 
cover of an elliptic curve with a certain trace~$t$.  Thus, one way to 
show that there are no genus-$4$ curves over $\FF_q$ with $N$ points
would be to enumerate all of the elliptic curve over $\FF_q$ of 
trace~$t$, enumerate all of the genus-$4$ double covers of these 
curves, count the number of points on the double covers, and verify 
that none of them has~$N$ points.  This is the strategy we use for the 
appropriate entries in Table~\ref{T:upper}.

In this section we explain more details of this procedure.  The actual
Magma programs we use can be found at the URL mentioned in the 
introduction; follow the link associated with this paper, and then 
download the file \texttt{Genus4.magma}.

Our first comment is that if $C$ is a double cover of an elliptic curve
$E$ over a finite field $k$, and if $\sigma$ is an automorphism of~$k$,
then $C^\sigma$ is a double cover of~$E^\sigma$.  Since $C^\sigma$ and 
$C$ have the same number of points, it will suffice for us to enumerate
all of the double covers of a set of representatives of the trace-$t$
elliptic curves up to automorphisms of~$k$.

The finite fields we must investigate are small enough that we use a 
completely naive method of finding representatives for these elliptic
curves.  We will only be working with fields of characteristic larger
than $3$, so every elliptic curve can be written in the form
$y^2 = x^3 + ax + b$, where $a$ and $b$ are elements of $k$ with
$4a^3 + 27b^2\ne 0$.  The curve determined by one such pair $(a,b)$ is 
isomorphic to a Galois conjugate of the curve determined by another
pair $(a',b')$ if and only if $(a',b') = (a^\sigma u^4, b^\sigma u^6)$
for some $u\in k^*$ and some automorphism $\sigma$ of~$k$.  The 
function \texttt{ECs} takes as input a finite field $k$ and a 
trace~$t$, explicitly computes the orbits of the set of $(a,b)$ pairs
under the combined action of $k^*$ and $\Aut k$, and computes the trace
of one representative elliptic curve from each orbit.  It returns those
representatives that have the desired trace~$t$.

Next, given an $E$ of trace~$t$, we must enumerate its genus-$4$ double
covers~$C$.  We use the same general idea as 
in~\cite{HoweLauter2003}*{\S6.1}; our description of the method is 
adapted from the version given there.

The function field of such a $C$ is obtained from that of $E$ by 
adjoining a root of $z^2 = f$, where $f$ is a function on $E$. By the
Riemann-Hurwitz formula, in order for $C$ to have genus $4$ the divisor
of $f$ must be of the form 
\[ P_1 + \cdots + P_6 + 2D, \]
where the $P_i$ are distinct geometric points on $E$ such that the
divisor $P_1 + \cdots + P_6$ is $k$-rational, and where $D$ is a
divisor of degree~$-3$.  There is a function $g$ on $E$ such that
\[ D + \divisor g = Q - 4\infty, \]
where $\infty$ is the infinite point on~$E$ and where $Q$ is a rational
point on~$E$. Replacing $f$ with $fg^2$ gives an isomorphic double 
cover of $E$. Thus, we may assume that $C$ is given by adjoining a root
of $z^2 = f$, where $f$ is a function on $E$ whose divisor is of the 
form 
\begin{equation}
\label{EQ:gooddiv}
P_1 + \cdots + P_6 + 2Q - 8\infty,
\text{\quad where the $P_i$ are distinct.
}
\end{equation}
  
We can also change the map $C\to E$ by following it with a translation
map on~$E$.  Translating  $E$ by a rational point $R$ has the effect of
replacing $f$ with a function whose divisor is 
\[ (P_1 + R) + \cdots + (P_6 + R) + 2(Q + R) - 8R \]
(where the sums in parentheses take place in the algebraic group~$E$).
By modifying this new $f$ by the square of a function we can get the 
divisor of $f$ to be 
\[ (P_1 + R) + \cdots + (P_6 + R) + 2(Q - 3R) - 8\infty. \]
We see that we only need to consider $Q$ that represent distinct 
classes of $E(k)$ modulo $3E(k)$.  Note that every class other than the
class of the identity element contains a representative that is not a
$2$-torsion point, so we may assume that $Q$ does not have order~$2$.
And finally, we note that we need only look at one representative from
each orbit of $\Aut E$ acting on~$E(k)/3E(k)$.  

Next, given an $E$ and a $Q\in E(k)$ not of order~$2$, we must 
enumerate all functions on $E$ with divisors of the
form~\eqref{EQ:gooddiv}.  Suppose first that $Q\ne\infty$.  By 
translating co\"ordinates on $E$, we may assume that $E$ is given by
an equation $y^2 = x^3 + r x^2 + s x + t^2$ and that $Q$ is the point
$(0,t)$.  Since $Q$ is not a $2$-torsion point we have $t\neq 0$, and
the tangent line to $E$ at $Q$ is given by $y = mx + t$, where 
$m = s/(2t)$.  Then there are two cases to consider: functions for
which none of the $P_i$ is equal to $\infty$, and functions for which
one of the $P_i$ is equal to~$\infty$.

If $f$ is a function on $E$ whose divisor has the desired form and for
which no $P_i$ is $\infty$, then $f$ has degree~$8$ and lies in the
Riemann-Roch space $\scrL(8\infty-2Q)$.  We check that this space is
spanned by the functions
\[ \{ x^4,\  x^2(y-t), \ x^3, \ x(y-t), \ x^2, \ y - mx - t \}. \]
We can then run through all of the $f$'s in the $k$-span of these 
functions, considering only those linear combinations where the 
coefficient of $x^4$ is nonzero (so that the function actually has 
degree~$8$).  In fact, since $z^2 = f$ and $z^2 = d^2 f$ give 
isomorphic extensions of $E$, we can restrict our attention to linear
combinations where the coefficient of $x^4$ is either $1$ or a fixed
nonsquare value.

For a given linear combination $f$, we can easily compute the number of
points on the extension $z^2 = f$ under the assumption that the divisor
of $f$ is of the form~\eqref{EQ:gooddiv}. (We get either $2$ or $0$ 
points over $\infty$ depending on the leading term of the Laurent 
expansion of $f$ at $\infty$, and likewise for the points over $Q$;
for the other points $P$ of $E(k)$, we get either $0$, $1$, or $2$ 
points over $P$ depending on whether $f(P)$ is a nonsquare, zero, or a
nonzero square.) If the number we calculate is larger than our previous
best count, we can then check whether the divisor of $f$ is in fact of
the form~\eqref{EQ:gooddiv}.

The case where one of the $P_i$ is equal to $\infty$ is very similar;
the difference is that we now look in the Riemann-Roch space
$\scrL(7\infty-2Q)$, and that now the extension $z^2=f$ always has 
exactly one point lying over $\infty$.

The function \texttt{double\us{}covers\us{}genus\us{}4} in the file 
\texttt{Genus4.magma} takes as input an elliptic curve $E$ over a
finite field, and runs the algorithm sketched above to find the largest
number of points on a genus-$4$ double cover of~$E$. The function 
\texttt{double\us{}covers\us{}given\us{}trace} takes as input a prime
power $q$ and a trace $t$, and finds the maximal number of points on a
genus-$4$ curve over $\FF_q$ that is a double cover of some elliptic
curve over $\FF_q$ of trace~$t$.

For each pair $(q,N)$ whose associated entry in Table~\ref{T:upper}
says that a genus-$4$ curve over $\FF_q$ with $N$ points must be a
double cover of an elliptic curve of trace $t$, we ran our program with
input $(q,t)$.  For each pair, we found that the maximal number of
points on a genus-$4$ double cover of an elliptic curve with trace $t$
is less than $N$.  Thus, for these pairs $(q,N)$, we find that
$N_q(4) < N$.

\section{Hermitian forms over maximal quadratic orders}
\label{S:maximal}

The programs in \texttt{IsogenyClasses.magma} show that for several of
the $(q,N)$ pairs listed in Table~\ref{T:upper}, every genus-$4$ curve
over $\FF_q$ having $N$ points must have Jacobian isogenous to~$E^4$,
where $E$ is an ordinary elliptic curve over $\FF_q$ whose Frobenius
endomorphism generates a specific imaginary quadratic order~$R$.  The
study of abelian varieties isogenous to $E^4$ is simplified by the use
of Serre's `Hermitian modules'~\cite{LauterSerre2002}*{Appendix} or
Deligne's equivalence of categories~\cite{Deligne1969} between ordinary
abelian varieties and modules over a certain ring, combined with the
description of polarizations on these modules provided 
in~\cite{Howe1995}.

In this section we analyze the three cases from Table~\ref{T:upper}
where the quadratic order in question is maximal; these are the cases
$(q,N) = (41,90)$, $(q,N) = (47,100)$, and $(q,N) = (61,122)$, where
the corresponding quadratic order $\scrO$ has discriminant $-20$, 
$-19$, and $-19$, respectively.  As is explained 
in ~\cite{LauterSerre2002}*{Appendix}, a principally-polarized abelian
variety isogenous to $E^4$ corresponds to a principally-polarized 
Hermitian $\scrO$-module of rank~$4$. Schiemann~\cite{Schiemann1998}
has computed all such principally-polarized Hermitian modules (up to
isomorphism) for the quadratic orders we are concerned with;  the lists
can be found at 

\centerline{\url{http://www.math.uni-sb.de/ag/schulze/Hermitian-lattices/}}

\noindent
as well as on the author's web site, mentioned in the introduction.

Schiemann calculated the automorphism groups of these polarized 
Hermitian modules; these groups are isomorphic to the automorphism
groups of the corresponding polarized abelian four-folds. Schiemann's
tables show that for all of the Hermitian modules we must consider, the
automorphism groups have order divisible by~$4$.  By Torelli's 
theorem~\cite{Milne1986}*{Thm.~12.1, p.~202}, if the polarized Jacobian
of a curve $C$ has an automorphism group of order divisible by~$4$,  
then either
\begin{itemize}
\item $C$ has an automorphism $\iota$ of order~$2$ such that the
      quotient curve $C_0=C/\langle\iota\rangle$ has positive genus, or
\item $C$ is a hyperelliptic curve with an automorphism of order~$4$
      whose square is the hyperelliptic involution.
\end{itemize}      
For the $(q,N)$ pairs we have to consider, we find that if $C$ is a
genus-$4$ curve over $\FF_q$ with $N$ points, then either $C$ is a
double cover of an elliptic curve isogenous to~$E$, or $C$ is a double
cover of a genus-$2$ curve whose Jacobian is isogenous to~$E^2$, or $C$
is a hyperelliptic curve with an automorphism whose square is the 
hyperelliptic involution.

We have already seen how to enumerate the genus-$4$ double covers of an 
elliptic curve; running \texttt{double\us{}covers\us{}given\us{}trace}
for the $(q,t)$ pairs coming from our $(q,N)$ pairs, we find no
genus-$4$ curves over $\FF_q$ with $N$ points.
In Section~\ref{SS:genus2} below we will show how to enumerate the 
genus-$4$ double covers of genus-$2$ curves, and in
Section~\ref{SS:hyperelliptic} we will show how to enumerate the genus-$4$
hyperelliptic curves with an automorphism whose square is the hyperelliptic involution.
For our $(q,N)$ pairs, we find no genus-$4$ curves over $\FF_q$ having $N$ points.

\subsection{Double covers of genus-$2$ curves}
\label{SS:genus2}
Suppose $\varphi:C\to C_0$ is a degree-$2$ map from a genus-$4$ curve
to a genus-$2$ curve over a finite field $k$ of characteristic greater
than~$2$.  The function field of $C$ is obtained from that of $C_0$ by
adjoining a root of $z^2 = f$, where $f$ is a function on $C_0$. By the
Riemann-Hurwitz formula, in order for $C$ to have genus $4$ the divisor
of $f$ must be of the form 
\[ P_1 + P_2 + 2D, \]
where the $P_i$ are distinct geometric points on $E$ such that the
divisor $P_1 + P_2$ is $k$-rational, and where $D$ is a divisor of 
degree~$-1$.  Let $\infty$ be any rational point on~$C_0$. The 
Riemann-Roch theorem shows that the Riemann-Roch space 
$\scrL(D + 4\infty)$ has dimension $2$; it follows that it must contain
a function $g$ such that $\divisor g + D + 3\infty$ is effective. Then
\[\divisor fg^2 = P_1 + P_2 + 2D' - 6\infty\]
for some effective divisor $D'$ of degree~$2$.  Replacing $f$ 
with~$fg^2$, we find that $C$ can be obtained from $C_0$ by adjoining a 
root of $z^2 = f$, with $\divisor f = P_1 + P_2 + 2D' - 6\infty$ for 
some effective $D'$ of degree~$2$.

Thus, to enumerate all genus-$4$ curves that are double covers 
of~$C_0$, we can simply enumerate all effective degree-$2$ 
divisors~$D'$, compute the Riemann-Roch space $\scrL(6\infty - 2D')$,
loop through the function $f$ in this space (up to squares in~$k$) that
are not squares in $\kbar(C_0)$, and consider the curves $z^2 = f$.  It
is easy to count the number of points on a curve defined by such an 
extension.  We implement this algorithm in the function 
\texttt{double\us{}covers\us{}genus\us{}4}, the same function we used
for doubles covers of elliptic curves;  when the input to the function
is a curve of genus~$2$, the function runs through the procedure
sketched above, and outputs the largest number of points it finds.  

For the case $(q,N) = (41,90)$ there are three possibilities for the 
base curve~$C_0$; this can be determined by brute force (by enumerating
all genus-$2$ curves over~$k$) or by theory, by noting that Schiemann's
tables show exactly three unimodular rank-$2$ Hermitian $\scrO$-modules
that are not products of two rank-$1$ modules.  The three curves are
\begin{align*}
y^2 &= x^6 + 7 x^4 + 8 x^2 - 7,\\
y^2 &= x^6 + 7 x^4 + 3 x^2 + 7, \quad\textup{and}\\
y^2 &= x^6 - 3 x^4 - 3 x^2 + 1.
\end{align*}
Running \texttt{double\us{}covers\us{}genus\us{}4} on these curves, we
find no genus-$4$ curves over $\FF_{41}$ with $90$ points.
Similarly, we find one possible $C_0$ for each of the other two $(q,N)$
pairs we must consider.  
For~$\FF_{47}$ and $\FF_{61}$, these curves are
\[y^2 = x^6 + 7 x^4 - 9 x^2 - 6 \quad\textup{and}\quad
y^2 = x^6 + 10 x^4 - 11 x^2 - 1,\]
respectively.
Running \texttt{double\us{}covers\us{}genus\us{}4} on these curves, we
find no genus-$4$ curves over $\FF_{47}$ with $100$ points, and no
genus-$4$ curves over $\FF_{61}$ with $122$ points.

\subsection{Hyperelliptic curves with automorphisms of order $4$}
\label{SS:hyperelliptic}

A hyperelliptic curve over $\FF_q$ can have at most $2(q+1)$ rational
points, so the only one of our three $(q,N)$ pairs that we need be
concerned with is $(q,N) = (61,122)$.

Suppose $C$ is a genus-$4$ hyperelliptic curve over a finite field $k$
with an automorphism $\alpha$ of order~$4$ whose square is the 
hyperelliptic involution. (In fact,~\cite{GuralnickHowe2009}*{Thm.~1.1}
shows that \emph{every} order-$4$ automorphism of a genus-$4$ 
hyperelliptic curve has square equal to the hyperelliptic involution.)
We can choose a parameter for $\PP^1$ so that the automorphism of 
$\PP^1$ induced by $\alpha$ is $x \mapsto a/x$ for some~$a\in k^*$.

Writing $C$ as $y^2 = f$ for some separable polynomial $f\in k[x]$ of
degree $9$ or~$10$, we see that the automorphism $\alpha$ must have the
form $(x,y) \mapsto (a/x, b y / x^5)$, and since $\alpha^2$ is the
hyperelliptic involution we must have $b^2 = -a^5$.  Replacing $x$ with
$a^2 x / b$, we compute that $\alpha$ is given by
$(x,y) \mapsto (-1/x,  y / x^5)$.

We find that $f$ must be a linear combination of the following 
polynomials:
\[ \{ x^{10} + 1  , \ 
      x^9    - x  , \ 
      x^8    + x^2, \ 
      x^7    - x^3, \
      x^6    + x^4  \}.\]
For the one pair $(q,N)$ we must consider, we can examine all linear
combinations of these polynomials (up to squares in~$\FF_q$), and
compute the number of points on the associated hyperelliptic curves.
We find no curve over $\FF_{61}$ having $122$ points.

The function \texttt{hyperelliptic} in the file \texttt{Genus4.magma}
implements this algorithm.

\section{Hermitian forms over nonmaximal quadratic orders}
\label{S:nonmaximal}

For five of the entries $(q,N)$ in Table~\ref{T:upper}, we see that the 
Jacobian of a genus-$4$ curve over $\FF_q$ with $N$ points must be 
isogenous to~$E^4$, where $E$ lies in an isogeny class of ordinary
elliptic curves over $\FF_q$ all of whose elements have complex 
multiplication by a non-maximal order $R$ in a quadratic field~$K$; 
here $R$ is the ring generated over $\ZZ$ by the Frobenius 
endomorphism of~$E$. The five entries are $(19, 52)$, $(29, 70)$, 
$(53, 110)$, $(67, 132)$, and $(71, 136)$, and the rings $R$ are the 
orders of conductor~$2$ inside the maximal orders of, respectively, 
$\QQ(\sqrt{-3})$, $\QQ(i)$, $\QQ(i)$, $\QQ(\sqrt{-3})$, and 
$\QQ(\sqrt{-7})$.  The discriminants of the rings $R$ are, 
respectively, $-12, -16, -16, -12,$ and~$-28$. We will show that for
each of these pairs $(q,N)$, a genus-$4$ curve over $\FF_q$ having $N$
points must be a double cover of a curve of genus $1$ or~$2$.

\begin{proposition}
\label{P:nonmaximal}
Let $t$ be the trace of Frobenius for an isogeny class $\scrC$ of
ordinary elliptic curves over a finite field~$\FF_q$ such that 
$\Delta := t^2 - 4q$ lies in $\{-12,-16,-28\}.$  Suppose $C$ is a
genus-$4$ curve over $\FF_q$ whose Jacobian is isogenous to the fourth
power of an elliptic curve in $\scrC$.
\begin{enumerate}
\item If $\Delta=-12$ or $\Delta=-16$ then $C$ is a double cover of a 
curve in~$\scrC$.
\item If $\Delta=-28$ then either $C$ is a double cover of a curve
in~$\scrC$, or $C$ is a double cover of a genus-$2$ curve whose 
Jacobian is isogenous to the square of a curve in~$\scrC$.
\end{enumerate}
\end{proposition}

For our five $(q,N)$ pairs, the function
\texttt{double\us{}covers\us{}given\us{}trace} finds no genus-$4$ curve
with $N$ points.  For the pair $(q,N) = (71, 136)$ we must also run
\texttt{double\us{}covers\us{}genus\us{}4} on three genus-$2$ curves:
\begin{align*}
y^2 &= x^6 - 29 x^4 - 29 x^2 + 1, \\
y^2 &= x^6 - 13 x^4 - 13 x^2 + 1, \quad\textup{and}\\
y^2 &= x^6 - 12 x^4 - 15 x^2 - 1.
\end{align*}
Again, we find no genus-$4$ curve with $N$ points.

The proof of Proposition~\ref{P:nonmaximal} relies on several lemmas.
To begin, we show that for each of our five cases, there are only four
abelian varieties to consider.

\begin{lemma}
\label{L:products}
Let $t$ be the trace of Frobenius for an isogeny class $\scrC$ of
ordinary elliptic curves over a finite field~$\FF_q$ such that 
$\Delta := t^2 - 4q$ lies in $\{-12,-16,-28\}.$  Then $\scrC$ contains
exactly two elliptic curves, one of them with endomorphism ring equal
to the maximal order $\scrO$ in $K:=\QQ(\sqrt{\Delta})$, and one with 
endomorphism ring equal to the order $R$ of conductor $2$ in $\scrO$.

Let $E$ and $F$ be these two elliptic curves, where $E$ has the larger
endomorphism ring.  Then every abelian variety over $\FF_q$ isogenous
to $E^n$ is isomorphic to $E^i \times F^{n-i}$ for some $i$, and the
varieties arising from distinct values of $i$ are not isomorphic to one
another.
\end{lemma}

\begin{proof}
By Deligne's equivalence of categories~\cite{Deligne1969}, the elliptic
curves in $\scrC$ correspond to the isomorphism classes of nonzero
finitely-generated $R$-modules in~$K$.  The endomorphism ring of such a
module is either $R$ or $\scrO$, and since both $R$ and $\scrO$ have 
class number $1$, there is one elliptic curve $E$ with 
$\End E \cong \scrO$ and one elliptic curve $F$ with $\End F \cong R$.

Likewise, Deligne's theorem shows that the varieties isogenous to $E^n$
correspond to rank-$n$ modules over $R$.  By using a result of Borevich
and Faddeev~\cite{BorevichFaddeev1965}, we see that every such module
is isomorphic to a sum of copies of $\scrO$ plus a sum of copies 
of~$R$, and that two such modules are isomorphic if and only if they
have the same number of each type of summand.  The second statement of
the lemma follows upon applying Deligne's equivalence of categories.
\end{proof}

Next, we show that for each of our cases, any curve with Jacobian
isomorphic to $E^i \times F^{4-i}$ with $i>0$ must be a double cover
of~$E$.

\begin{lemma}
\label{L:pullback}
Let $t$ be the trace of Frobenius for an isogeny class of ordinary 
elliptic curves over a finite field~$\FF_q$ such that 
$\Delta := t^2 - 4q$ lies in $\{-12,-16,-28\},$ and let $E$ and $F$ be
as in Lemma~\textup{\ref{L:products}}.  If $C$ is a genus-$4$ curve
over $\FF_q$ whose Jacobian is isomorphic to $E^i\times F^{4-i}$ for
some $i>0$, then $C$ is a double cover of~$E$.
\end{lemma}

\begin{proof}
Our proof is computational, and uses the function 
\texttt{pullback\us{}bound} from the package 
\texttt{IsogenyClasses.magma}.  Given a fundamental discriminant $D$
with $-11\le D < 0$, an integer $n$, and a dimension~$d$, this function
will return an integer $B$ such that for every $d\times d$ Hermitian
matrix of determinant $n$ with entries in the quadratic order $\scrO_D$
of discriminant $D$, the Hermitian form on $\scrO_D^d$ determined by
this matrix will have a nonzero vector of length at most~$B$.  As is
explained in~\cite{HoweLauter2012}*{\S4}, this translates into a 
statement about polarizations on $E^d$, where $E$ has complex
multiplication by $\scrO_D$; namely, for every degree-$n^2$ 
polarization $\lambda$ on $E^d$, there is an embedding 
$\varphi\col E\to E^d$ such that $\varphi^*\lambda$ is a polarization
on $E$ of degree at most~$B^2$.

Now suppose $C$ is a curve with Jacobian isomorphic to 
$E^i\times F^{4-i}$, with $i>0$, and let $\lambda$ be the canonical
principal polarization on $\Jac C$.  There is an isogeny of degree
$2^{4-i}$ from $E^4$ to $\Jac C$, and pulling back $\lambda$ via this
isogeny gives a polarization $\mu$ of degree $4^{4-i}$ on $E^4$.  The
function \texttt{pullback\us{}bound} tells us that we can pull back
$\mu$ to a polarization of degree $1$ or $4$ on $E$, and 
by~\cite{HoweLauter2012}*{Lem.~4.3} this gives us a map of degree $1$
or $2$ from $C$ to $E$.  A map of degree~$1$ is clearly impossible, so
we find that $C$ is a double cover of~$E$.
\end{proof}

We see that to prove Proposition~\ref{P:nonmaximal} we need only
consider curves whose Jacobians are isomorphic to~$F^4$; in other
words, we need only consider principal polarizations on~$F^4$.  For
$\Delta = -12$ and $\Delta = -16$ we will show that every such 
polarization can be pulled back to either $E$ or $F$ to get a
polarization of degree $1$ or $4$.  For $\Delta = -28$ we will show
that if such a polarization cannot be pulled back to a polarization of
degree $1$ or $4$ on $E$ or~$F$, and if the polarized variety is the
Jacobian of a curve, then the curve has an involution that makes it a
double cover of a genus-$2$ curve.

Let $\varphi$ be a degree-$2$ isogeny from $E$ to $F$, and let 
$\Phi\col E^4 \to F^4$ be the degree-$16$ product isogeny 
$\varphi\times\varphi\times\varphi\times\varphi$.  Denote the kernel of
$\Phi$ by $G$.  Note that the smallest $\scrO$-stable subgroup of $E^4$
that contains $G$ is $E^4[2]$.

Suppose $\lambda$ is a principal polarization on $F^4$, and let
$\mu = \Phi^*\lambda$ be the pullback of $\lambda$ to $E^4$. Then $\mu$
is a polarization of degree $16^2$ on $E^4$, and since $\ker \mu$ is 
stable under the action of $\scrO$ and contains $G$, we must have 
$\ker\mu = E^4[2]$.  That means that $\mu$ must be twice a principal
polarization.

Thus, to enumerate the possible principal polarizations $\lambda$ 
on~$F^4$, we can enumerate the principal polarizations on $E^4$ up to
isomorphism, multiply each of these by $2$ to get a polarization $\mu$
of degree $16^2$ on $E^4$, enumerate the subgroups $G$ of $E^4[2]$ that
generate all of $E^4[2]$ as an $\scrO$-module and that are isotropic 
with respect to the Weil pairing on $E^4[2]$ determined by $\mu$, and 
then compute the matrix in $M_4(R)$ that represents the polarization we
get by pushing $\mu$ down through an isogeny with kernel $G$.  The 
Magma programs we use to do this can be found at the URL mentioned in 
the introduction; follow the link associated with this paper, and then 
download the file \texttt{NonMaximalOrders.magma}.

\subsection{Proof of Proposition~\ref{P:nonmaximal} when $\Delta=-12$}
\label{SS:Delta12}

Schiemann's tables~\cite{Schiemann1998} show that up to isomorphism the
only principal polarization on $E^4$ is the product polarization, so
our only $\mu$ is represented by the diagonal matrix with $2$'s on the
diagonal. Embedding $E$ into $E^4$ along one of the factors, we find
that~$\mu$, and hence~$\lambda$, can be pulled back to a polarization 
of degree $4$ on~$E$.

\subsection{Proof of Proposition~\ref{P:nonmaximal} when $\Delta=-16$}
\label{SS:Delta16}

Schiemann's tables show two principal polarizations on $E^4$, the 
product polarization and the polarization given by the Hermitian matrix
\[
P = \left[ \ 
\begin{matrix}
       2   &     1   &     0  &      0 \\
       1   &     2   &   1-i  &    1-i \\
       0   &   1+i   &     2  &      1 \\
       0   &   1+i   &     1  &      2 
\end{matrix}       
\ \right].
\]
Twice the product polarization can be pulled back to give a 
polarization of degree $4$ on $E$, so we need only consider the
polarization~$2P$.

A computer calculation shows that there are $1024$ subgroups of
$E^4[2]$ that are maximal isotropic with respect to the Weil pairing
determined by $P$ and that generate all of $E^4[2]$ as an 
$\scrO$-module.  For each such subgroup $G$, we compute the
polarization on $F^4$ obtained by pushing the polarization $2P$ down
through the isogeny with kernel $G$; this polarization can be 
represented by a Hermitian matrix in $M_4(R)$, which gives a Hermitian
form on $R^4$.  We can then compute the short vectors in this Hermitian
lattice.  We find that for each group $G$, the Hermitian lattice has a
short vector of length $2$.  This means that the polarization on $F^4$
can be pulled back to a polarization of degree $4$ on $F$.

\subsection{Proof of Proposition~\ref{P:nonmaximal} when $\Delta=-28$}
\label{SS:Delta28}

Schiemann's tables show that there are three principal polarizations
on $E^4$: the product polarization and the polarizations given by the 
Hermitian matrices
\begin{align*}
P_1 &= \left[ \ 
\begin{matrix}
 1  &             0            &              0            &        0                  \\
 0  &             2            &  \frac{1 + \sqrt{-7}}{2}  &       -1\phantom{-}                  \\
 0  & \frac{1 - \sqrt{-7}}{2}  &              2            &  \frac{1 + \sqrt{-7}}{2}  \\
 0  &       -1\phantom{-}      &  \frac{1 - \sqrt{-7}}{2}  &        2
\end{matrix}
\ \right]
\\
\intertext{and}
P_2 &= \left[ \ 
\begin{matrix}
             2            &             1            &              0            &  \frac{1 + \sqrt{-7}}{2}  \\
             1            &             2            &  \frac{1 + \sqrt{-7}}{2}  &  \frac{1 + \sqrt{-7}}{2}  \\
             0            & \frac{1 - \sqrt{-7}}{2}  &              2            &              1            \\
 \frac{1 - \sqrt{-7}}{2}  & \frac{1 - \sqrt{-7}}{2}  &              1            &              2            \\
\end{matrix}
\ \right].
\end{align*}

There is a $2$ on the diagonal of $2P_1$, so this polarization can be 
pulled back to a degree-$4$ polarization on $E$.  Likewise, twice the
product polarization can be pulled back to a degree-$4$ polarization 
on~$E$.  That leaves us to consider the polarization~$2P_2$. 

A computer calculation shows that there are $448$ subgroups of $E^4[2]$
that are maximal isotropic with respect to the Weil pairing determined
by $P_2$ and that generate all of $E^4[2]$ as an $\scrO$-module.  Of
these subgroups, $256$ give rise to principal polarizations on $F^4$
whose associated Hermitian forms have short vectors of length~$2$. The
other $192$ subgroups give principal polarizations on $F^4$ that are 
isomorphic to the polarizations defined by one of the following three 
Hermitian matrices:
\begin{align*}
Q_1 &=\left[ \ 
\begin{matrix}
                   4             &                   2             &         2             & -\sqrt{-7}\phantom{-} \\
                   2             &                   4             & -\sqrt{-7}\phantom{-} &         2             \\
                   2             & \phantom{-}\sqrt{-7}\phantom{-} &         4             &        -2\phantom{-}  \\
 \phantom{-}\sqrt{-7}\phantom{-} &                   2             &        -2\phantom{-}  &         4             \\
\end{matrix}\ \right]\displaybreak[0]\\
Q_2 &= \left[ \ 
\begin{matrix}
  3             &  1             &  1             & -1 - \sqrt{-7} \\
  1             &  3             & -1 - \sqrt{-7} &  1             \\
  1             & -1 + \sqrt{-7} &  4             & -2\phantom{-}  \\
 -1 + \sqrt{-7} &  1             & -2\phantom{-}  &  4             \\
\end{matrix}\ \right]\displaybreak[0]\\
Q_3 &= \left[ \ 
\begin{matrix}
                   3             &                   1             &         0             & -\sqrt{-7}\phantom{-} \\
                   1             &                   3             & -\sqrt{-7}\phantom{-} &         0             \\
                   0             & \phantom{-}\sqrt{-7}\phantom{-} &         3             &        -1\phantom{-}  \\
 \phantom{-}\sqrt{-7}\phantom{-} &                   0             &        -1\phantom{-}  &         3             \\
\end{matrix}\ \right] .
\end{align*}
We check that for each of these three polarizations $Q$ there is an
involution of the polarized variety $(F^4,Q)$ that fixes a 
$2$-dimensional subvariety of~$F^4$; such an involution can be 
represented by a matrix $A$ such that $A^2 = I$ and $A^* Q A = Q$, and
such that $A-I$ has rank~$2$. For each $Q$ we note that the matrix
\[
A = \left[ \ 
\begin{matrix}
 0 & 1 & 0 & 0 \\
 1 & 0 & 0 & 0 \\
 0 & 0 & 0 & 1 \\
 0 & 0 & 1 & 0 \\
\end{matrix}\ \right]
\]
has these properties. It follows from Torelli's 
theorem~\cite{Milne1986}*{Thm.~12.1, p.~202} that if $(F^4,Q)$ is the 
Jacobian of a curve $C$, then $C$ has an involution $\alpha$ such that
the quotient of $C$ by $\alpha$ is a genus-$2$ curve whose Jacobian is
necessarily isogenous to~$F^2$.

This completes the proof of Proposition~\ref{P:nonmaximal}.\qed
\section{Hermitian forms over $\ZZ[\zeta_5]$}
\label{S:zeta}

Let $\scrO = \ZZ[\zeta_5]$.  As Table~\ref{T:upper} indicates, the
programs in \texttt{IsogenyClasses.magma} show that if $C$ is a 
genus-$4$ curve over $\FF_{11}$ with $34$ points, the center of the
endomorphism ring of the Jacobian of $C$ is isomorphic to $\scrO$.  (At
the end of this section we will explain how to derive this statement
from the output of the programs.) Likewise, if $C$ is a genus-$4$ curve
over $\FF_{61}$ with $120$ points that is not a double cover of an
elliptic curve, then the center of the endomorphism ring of its
Jacobian is isomorphic to $\scrO$.  In this section, we show that no
ordinary genus-$4$ curve over a finite field can have $\scrO$ as the
center of the endomorphism ring of its Jacobian.

\begin{proposition}
\label{P:zeta5}
Let $k$ be a finite field.  There is no genus-$4$ curve over $k$ with
ordinary Jacobian $J$ such that the center of $\End J$ is isomorphic
to~$\ZZ[\zeta_5]$.
\end{proposition}

\begin{proof}
Suppose, to obtain a contradiction, that such a curve $C$ exists.  Let
$K = \QQ(\zeta_5)$, let $\pi$ be the Frobenius endomorphism of~$J$, let
$\pibar$ be the Verschiebung endomorphism of $J$, and let $R$ be the
ring $\ZZ[\pi,\pibar]$.  Since $R$ lies in the center of $\End J$ we
may view $R$ as a subring of $\scrO$, and since $\pi$ generates the
center of $(\End J)\otimes\QQ$, we see that $R$ is an order in~$\scrO$.
Let $M$ be the Deligne module associated with $J$ 
(see~\cite{Howe1995}), so that $M$ is a finitely-generated $R$-module
that is isomorphic to a submodule of~$K^2$.  Since $\End J$ contains
$\scrO$, we see that $M$ is in fact an $\scrO$-module, and since 
$\scrO$ has class number $1$ we have $M \cong \scrO\oplus\scrO$.  
Translating this statement back from Deligne modules to ordinary 
abelian varieties, we see that $J\cong A\times A$ for a $2$-dimensional
abelian variety with $\End A \cong \scrO$.

The field $K$ is ramified over its real subfield at a finite prime, so
by~\cite{Howe1995}*{Cor.~11.4, p.~2391} the variety $A$ has a principal
polarization $\lambda$.   Let $\mu$ be the canonical polarization 
on~$J$, viewed as a polarization on $A\times A$.  Then $\mu$ is equal
to the product polarization $\lambda\times\lambda$ preceded by an
endomorphism $P$ of $A\times A$; viewing $P$ as an element of
$M_2(\scrO)$, we find that $P$ is a unimodular Hermitian matrix that is
totally positive (meaning that all of the roots of its minimal
polynomial are totally positive algebraic numbers).  But 
Lemma~\ref{L:zeta5} below says that all totally positive rank-$2$ 
unimodular Hermitian lattices over $\scrO$ are decomposable, so $\mu$
is isomorphic to the product polarization $\lambda\times\lambda$. 
This contradicts the fact that the polarization on a Jacobian is never
a product.
\end{proof}

\begin{lemma}
\label{L:zeta5}
Let $\scrO = \ZZ[\zeta_5]$.  Suppose $P$ is a totally positive 
unimodular Hermitian matrix in $M_2(\scrO)$.  Then there is an
invertible $C\in M_2(\scrO)$ such that $P = C^*C$, where $C^*$ is
the conjugate transpose of~$C$.
\end{lemma}

\begin{proof}
Our proof follows the lines set out in~\cite{HoweLauter2003}*{\S8}.

Let $K = \QQ(\zeta_5)$, let $\varphi$ be a fundamental unit of the
maximal real subfield $\Kp = \QQ(\sqrt{5})$ of $K$ with 
$\Tr_{\Kp/\QQ} \varphi = 1$, and let $\phi$ be the real number 
$(1 + \sqrt{5})/2$.  Let $\psi_1$ and $\psi_2$ be embeddings of $K$ 
into the complex numbers $\CC$ with $\psi_1(\varphi) =\phi$ and
$\psi_2(\varphi) = 1/\phi$.  If $z$ is a complex number, we let $|z|$ 
be its magnitude and we let $||z||$ be its norm, so that 
$||z|| = |z|^2 = z \zbar$.

Let $q$ be the quadratic form on $\scrO$ that sends $x$ to the trace
from $K$ to $\QQ$ of~$x\xbar$, so that 
$q(x) = 2||\psi_1(x)|| + 2||\psi_2(x)||$.  Let $\Lambda$ be the lattice
$(\scrO,q)$.  Using Magma, we compute the Voronoi cell for this 
lattice, and we find that the covering radius of the lattice is~$2$;
that is, every element of $\Lambda\otimes\QQ$ differs from a lattice
point by an element $x$ with $q(x)\le 2$.  This means that every
element of $K$ differs from an element of $\scrO$ by an element $x$
satisfying
\[  ||\psi_1(x)|| + ||\psi_2(x)|| \le 1 . \]
In particular, we see that this $x$ also satisfies
\[  N_{K/\QQ}(x) \le 1/4 \text{\quad and \quad} ||\psi_i(x)|| \le 1
    \text{\quad for $i = 1,2$.} \]
In turn, this statement about the lattice $\Lambda$ gives us a
Euclidean algorithm on $\scrO$: Given elements $n$ and $d$ of~$\scrO$,
there are elements $q$ and $r$ of $\scrO$ such that $n = qd + r$ with
\[
        N(r) \le 1/4             \text{\quad and \quad } 
||\psi_i(r)|| \le ||\psi_i(n)||  \text{\quad for $i = 1,2 $}.
\]

Write our totally positive unimodular Hermitian matrix $P$ as
\[
P = 
\left[\ \begin{matrix}
  \alpha & \betabar \\
  \beta  & \gamma
\end{matrix}\ \right]
\]
where $\alpha$ and $\gamma$ are totally real and where $\alpha$ and 
$\alpha\gamma - \beta\betabar$ are totally positive.  We will 
repeatedly choose invertible matrices $C$ and replace $P$ with $C^*PC$
in order to reduce the size of the norm of the upper left element 
of~$P$.

The determinant of $P$ is a totally positive unit in $\Kp$, and so is
an even power of $\varphi$. By modifying $P$ by a matrix $C$ of the
form 
\[
\left[\ \begin{matrix}
  \varphi^i & 0 \\
     0      & 1
\end{matrix}\ \right]
\]
we may assume that $P$ has determinant~$1$. Then by modifying $P$ by a
power of the matrix 
\[
\left[\ \begin{matrix}
  \varphi & 0            \\
     0    & \varphi^{-1}
\end{matrix}\ \right]
\]
we can ensure that 
\[
\frac{1}{\phi^2} \le \frac{\psi_1(\alpha) }{ \psi_2(\alpha)} \le  \phi^2 . 
\]
Another way of expressing this is to say that 
\begin{equation} 
\label{EQ:alphabound} 
\frac{1}{\phi^2} \le \frac{\psi_i(\alpha)^2 }{\Norm_{\Kp/\QQ}(\alpha)}  
                 \le \phi^2 \quad\text{for $i = 1, 2$.} 
\end{equation}

Applying our Euclidean algorithm to $\beta$ and $\alpha$, we find that
$\beta = q \alpha + r$ for a $q$ and an~$r$ with 
$||\psi_i(r)|| \le ||\psi_i(\alpha)||$ and with 
$\Norm_{K/\QQ}(r) \le (1/4) \Norm_{K/\QQ}(\alpha)$.  If we set 
\[
C= 
\left[\ \begin{matrix}
  1 & -\qbar \\
  0 &  1
\end{matrix}\ \right]
\]
then 
\[
C^* P C  = 
\left[\ \begin{matrix}
  \alpha & \rbar\\
  r      & \gamma'
\end{matrix}\ \right]
\]
for some integer $\gamma'$ in $K^+$.   Replace $\beta$ with $r$ and 
$\gamma$ with~$\gamma'$, so that now we have
\begin{equation}
\label{EQ:betabound}
||\psi_i(\beta)|| \le ||\psi_i(\alpha)||  \quad\text{for $i=1,2$}
\end{equation}
and 
\begin{equation}
\label{EQ:betabound2}
\Norm_{K/\QQ}(\beta) \le (1/4) \Norm_{K/\QQ}(\alpha).
\end{equation}

Let $B = \beta \betabar$, so that $B$ is an integer in $\Kp$.  Note 
that we have $\alpha\gamma - B = 1$, so
\[
\psi_i(\alpha) \psi_i(\gamma) = 1 + \psi_i(B) \quad \text{for $i = 1,2 $}
\]
and therefore
\begin{equation} 
\label{EQ:gammabound}
\psi_i(\gamma) / \psi_i(\alpha) = 
   1/\psi_i(\alpha)^2 + \psi_i(B)/\psi_i(\alpha)^2 \text{\quad for $i = 1,2$.} 
\end{equation}

Now  let 
\begin{align*}
b_1 &= \psi_1(B) / \psi_1(\alpha^2) \\
b_2 &= \psi_2(B) / \psi_2(\alpha^2) \\
c_1 &= 1/\psi_1(\alpha^2)\\
c_2 &= 1/\psi_2(\alpha^2)
\end{align*}
so that equation~(\ref{EQ:gammabound}) becomes
\[
\psi_1(\gamma) / \psi_1(\alpha)  = b_1 + c_1 \text{\qquad and\qquad }
\psi_2(\gamma) / \psi_2(\alpha)  = b_2 + c_2.
\]
Multiplying these last two equalities gives
\begin{equation}
\label{EQ:gammaalphabound}
\Norm_{\Kp/\QQ}(\gamma/\alpha) = b_1 b_2 + b_1 c_2 + b_2 c_1 + c_1 c_2. 
\end{equation} 
Note that 
\begin{equation} 
\label{EQ:bbbound}
  b_1 b_2 = \frac{\Norm_{\Kp/\QQ}(B)}  {\Norm_{\Kp/\QQ}(\alpha)^2}
          = \frac{\Norm_{K/\QQ}(\beta)}{\Norm_{K/\QQ}(\alpha)} 
        \le \frac{1}{4}
\end{equation}
(where the final inequality comes from~(\ref{EQ:betabound2})) and 
\begin{equation}
\label{EQ:ccbound}
  c_1 c_2 = \frac{1}{\Norm_{\Kp/\QQ}(\alpha^2)}.
\end{equation}
Furthermore, from inequality~(\ref{EQ:alphabound}) we see that 
\begin{equation} 
\label{EQ:cbound} 
c_1  \le \frac{\phi^2}{\Norm_{\Kp/\QQ}(\alpha)} \qquad\text{and}\qquad 
c_2  \le \frac{\phi^2}{\Norm_{\Kp/\QQ}(\alpha)},
\end{equation}
and from inequality~(\ref{EQ:betabound}) we see that 
\begin{equation} 
\label{EQ:bbound} 
b_1 = \frac{||\psi_1(\beta)||}{||\psi_1(\alpha)||} \le 1 
\qquad\text{and}\qquad
b_2 = \frac{||\psi_2(\beta)||}{||\psi_2(\alpha)||} \le 1. 
\end{equation}

If we view $b_1$, $b_2$, $c_1$, and $c_2$ as non-negative real 
variables subject only to the conditions expressed in 
relations~(\ref{EQ:bbbound}), (\ref{EQ:ccbound}), (\ref{EQ:cbound}), 
and~(\ref{EQ:bbound}), and if we maximize $b_1c_1 + b_2c_2$ subject to
these conditions, we find that the maximum value occurs when $b_1 = 1$
and $c_1 = \phi^2/\Norm_{\Kp/\QQ}(\alpha)$.  Thus we have 
\begin{equation} 
\label{EQ:crosstermbound}
b_1c_1 + b_2c_2 \le 1\cdot \frac{\phi^2}{\Norm_{\Kp/\QQ}(\alpha)}
       + \frac{1}{4} \cdot \frac{(1/\phi^2)}{\Norm_{\Kp/\QQ}(\alpha)}
         \le \frac{2.72}{\Norm_{\Kp/\QQ}(\alpha)}. 
\end{equation} 
Let $\epsilon = 1/\Norm_{\Kp/\QQ}(\alpha)$.  Then by combining the 
relations~(\ref{EQ:gammaalphabound}), (\ref{EQ:bbbound}), 
(\ref{EQ:ccbound}) and~(\ref{EQ:crosstermbound}) we find that 
\[
\Norm_{\Kp/\QQ}(\gamma/\alpha) \le \epsilon^2  + 2.72 \, \epsilon + 1/4.
\]

If $\Norm_{\Kp/\QQ}(\alpha) \ge 4$ then $\epsilon \le 1/4$ and 
$\Norm_{\Kp/\QQ}(\gamma/\alpha) < 1$.  Then we can modify $P$ by 
\[
\left[\ \begin{matrix}
  0 & 1 \\
  1 & 0 \end{matrix}\ \right]
\]
to exchange $\alpha$ and $\gamma$, and this decreases the norm of the 
upper left element of $P$.

We repeat this procedure until we reach the point where
$\Norm_{\Kp/\QQ}(\alpha) \le 3$.  The only totally positive integer of
$\Kp$ with norm less than $4$ is $1$, so $\alpha=1$ and we can reduce
$\beta$ to be~$0$.  Then we find $\gamma = 1$, so that we have reduced 
$P$ to the identity matrix.
\end{proof}

Let us turn to the specific pairs $(q,N)$ that we must consider.  For
$q=11$ and $N=34$, the function \texttt{isogeny\us{}classes} in the
package \texttt{IsogenyClasses.magma} shows that a genus-$4$ curve $C$
over $\FF_{11}$ with $34$ points must have real Weil polynomial equal
to $(x^2 + 11x + 29)^2$; that is, the characteristic polynomial of
Frobenius plus Verschiebung is the square of this polynomial.  It 
follows that the characteristic polynomial of Frobenius is 
\[ (x^4 + 11 x^3 + 51 x^2 + 121 x + 121)^2, \]
which means that the Jacobian of $C$ is isogenous to the square of an
abelian surface $A$ with characteristic polynomial
$x^4 + 11 x^3 + 51 x^2 + 121 x + 121$.  Since the middle coefficient of
this characteristic polynomial is coprime to $q = 11$, we see that $A$
is ordinary.  Furthermore, the polynomial has a root $\pi$ in 
$\QQ(\zeta_5)$, namely $\pi = \zeta_5^2 + 2 \zeta_5 - 2$.  One checks
that $\pi$ and its complex conjugate generate the ring $\ZZ[\zeta_5]$,
so the center of $\End J$, which contains Frobenius and Verschiebung,
must equal $\ZZ[\zeta_5]$.  As we have seen, this is impossible.

For $q = 61$ and $N=120$, the function \texttt{isogeny\us{}classes}
tells us that if a genus-$4$ curve $C$ over $\FF_{61}$ with $120$
points is not a double cover of an elliptic curve of trace $-13$, then
it must have real Weil polynomial equal to $(x^2 + 29 x + 209)^2$.  As
above, we see that this implies that $C$ is ordinary and the center of
the endomorphism ring of its Jacobian is $\ZZ[\zeta_5]$, which is 
impossible.

\begin{remark}
\label{R:simpler}
For our particular cases $(q,N) = (11, 34)$ and $(q,N) = (61,120)$,
there is also a more computational approach to showing that there is no
genus-$4$ curve $C$ over $\FF_q$ with $N$ points.  As we have seen, we
may assume that there is a fifth root of unity in the center of the 
endomorphism ring of the Jacobian of $C$, and it follows that $C$ has
an automorphism of order~$5$.  Since the finite fields we are concerned
with contain the fifth roots of unity, this shows that $C$ is a
degree-$5$ Kummer extension of another curve, and by using the 
Riemann-Hurwitz formula we see that this second curve must be the 
projective line.  It is not hard to enumerate such Kummer covers; doing
so, we find no curves with $N$ points.
\end{remark}

\section{Lower bounds from explicit examples}
\label{S:lower}

In this section we prove the new lower bounds for $N_q(4)$ given in
Table~\ref{T:results} by providing examples of genus-$4$ curves with
many points.  Each line of Table~\ref{T:lower} gives a prime power~$q$,
an integer~$N$, and the equations for a genus-$4$ curve over $\FF_q$
having $N$ points.

\begin{table}
\renewcommand{\arraystretch}{1.25}
\begin{center}
\begin{tabular}{|r|r|lll|}
\hline
 $q$ &   $N$ & \multicolumn{3}{l|}{Equations for a genus-$4$ curve over $\FF_q$ with $N$ points} \\
\hline
$13$ &  $38$ & $y^2 = x^3 + 4$                        && $z^2 = x^3 + x^2 - 4 x - 3$             \\
$17$ &  $46$ & $y^2 = x^3 + x + 8$                    && $z^2 = x^3 - 5 x^2 - 2 x - 8$           \\
$19$ &  $48$ & $y^2 = x^3 + 8$                        && $z^2 = x^3 - 9 x^2 - 5$                 \\
$23$ &  $57$ & $y^2 = x^3 - 6 x^2 - 3 x - 7$          && $z^2 = x y + 2 x^3 - 11 x^2 + 7 x + 1$  \\
$29$ &  $67$ & $y^2 = x^3 - 3 x + 18$                 && $z^2 = y + 7 x^3 + 6 x^2 + 2 x - 10$    \\
$31$ &  $72$ & $y^2 = x^3 + x + 10$                   && $z^2 = x^3 + 7 x^2 - 13 x - 13$         \\
$37$ &  $82$ & $y^2 = x^3 + 2 x$                      && $z^2 = x^3 + x^2 - 10 x - 13$           \\
$41$ &  $88$ & $y^2 = x^3 + 6 x + 5$                  && $z^2 = 3 x^3 - 17 x^2 - 11 x$           \\
$43$ &  $92$ & $y^2 = x^3 + 2 x + 1$                  && $z^2 = x^3 - 6 x^2 + 11 x$              \\
$47$ &  $98$ & $y^2 = x^5 - 6 x^3 + 8 x^2 - 5 x + 12$ && $z^2 = y + 6 x^3 + 6 x^2 - x - 3$       \\
$49$ & $102$ & $y^2 = x^3 + x$                        && $z^2 = x^3 + x^2 + x + 4$               \\
$53$ & $108$ & $y^2 = x^3 + 4 x + 10$                 && $z^2 = x^3 + 12 x^2 + 17 x + 9$         \\
$59$ & $116$ & $y^2 = x^3 + 2 x + 22$                 && $z^2 = 2 x^3 + x^2 - x + 9$             \\
$61$ & $118$ & $y^2 = x^3 + 4$                        && $z^2 = x^3 + 23 x^2 + 25 x + 36$        \\
$67$ & $129$ & $y^2 = x^3 + 25$                       && $z^2 = x y - 8 x^3 - 27 x + 4$          \\
$71$ & $134$ & $y^2 = x^3 + x + 9$                    && $z^2 = x^3 + 9 x^2 + 24 x - 9$          \\
$73$ & $138$ & $y^2 = x^3 + 3 x + 11$                 && $z^2 = x^3 + 34 x^2 + 18 x + 40$        \\
$79$ & $148$ & $y^2 = x^3 + x + 6$                    && $z^3 = y + 33 x + 2$                    \\
$83$ & $152$ & $y^2 = x^3 + 2 x + 19$                 && $z^2 = x^3 + 38 x^2 - 6 x + 39$         \\
$89$ & $160$ & $y^2 = x^3 + 3 x$                      && $z^2 = x^3 + 13 x^2 - 22 x + 28$        \\
$97$ & $174$ & $y^2 = x^3 + 5 x + 26$                 && $z^3 = y + 37 x + 16$                   \\
\hline
\end{tabular}
\end{center}
\vskip0.5em
\caption{Genus-$4$ curves over small finite fields with many points.
         Note that for $q=79$ and $q=97$ the exponent on $z$ is $3$, not~$2$.}
\label{T:lower}
\end{table}

Almost all of the examples in Table~\ref{T:lower} were obtained by 
running our program \texttt{double\us{}cover\us{}given\us{}trace}.  The
exceptions are the example for $q=47$, which was found by running 
\texttt{double\us{}cover\us{}genus\us{}4} on a carefully chosen 
genus-$2$ curve, and the examples for $q=79$ and $q=97$, which were 
found during a search of degree-$3$ Kummer covers of elliptic curves.
While searching for such Kummer covers, we also found a particularly 
nice example for $q = 67$:  the curve $y^6 = x^3 + x - 6$ attains 
$N_{67}(4)$.

The function \texttt{check\us{}examples} in the file 
\texttt{Genus4.magma} verifies that all of these examples do have the
number of points claimed.






\end{document}